\documentclass[12pt]{amsart}

\usepackage[a4paper,margin=1in]{geometry}
\usepackage[T1]{fontenc}
\usepackage{amsmath,amssymb,amsthm,mathtools}
\usepackage{microtype}
\usepackage{xcolor}
\definecolor{linkblue}{RGB}{0,82,155}
\usepackage[
  colorlinks=true,
  linkcolor=linkblue,
  citecolor=linkblue,
  urlcolor=linkblue
]{hyperref}
\usepackage[nameinlink,capitalise,noabbrev]{cleveref}
\usepackage{enumitem}

\crefname{theorem}{Theorem}{theorems}
\Crefname{theorem}{Theorem}{Theorems}

\crefname{lemma}{lemma}{lemmas}
\Crefname{lemma}{Lemma}{Lemmas}

\crefname{corollary}{corollary}{corollaries}
\Crefname{corollary}{Corollary}{Corollaries}

\crefname{proposition}{proposition}{propositions}
\Crefname{proposition}{Proposition}{Propositions}

\crefname{remark}{remark}{remarks}
\Crefname{remark}{Remark}{Remarks}

\crefname{definition}{definition}{definitions}
\Crefname{definition}{Definition}{Definitions}

\selectfont
\setlist{itemsep=0.25em,topsep=0.4em}

\usepackage{aliascnt}

\newtheorem{theorem}{Theorem}[section]

\newaliascnt{lemma}{theorem}
\newtheorem{lemma}[lemma]{Lemma}
\aliascntresetthe{lemma}

\newaliascnt{corollary}{theorem}

\aliascntresetthe{corollary}

\newaliascnt{proposition}{theorem}

\aliascntresetthe{proposition}

\newtheoremstyle{myremark}%
  {6pt}
  {6pt}
  {\normalfont}
  {}
  {\bfseries}
  {.}
  {.5em}
  {}

\theoremstyle{myremark}
\newaliascnt{remark}{theorem}
\newtheorem{remark}[remark]{Remark}
\aliascntresetthe{remark}

\newaliascnt{definition}{theorem}
\newtheorem{definition}[definition]{Definition}
\aliascntresetthe{definition}

\DeclareMathOperator{\tr}{tr}

\newcommand{\R}{\mathbb{R}}
\newcommand{\1}{\mathbf{1}}
\newcommand{\inner}[2]{\langle #1,#2\rangle}
\newcommand{\froinner}[2]{\langle #1,#2\rangle_{\!F}}
\newcommand{\norm}[1]{\lVert #1\rVert}
\newcommand{\abs}[1]{\lvert #1\rvert}
\newcommand{\cP}{\mathcal{P}}

\title{Upper bound on the $k$-th eigenvalue of a graph}

\author{Varun Sivashankar}
\thanks{\textsc{Princeton University. Email:} \href{mailto:varunsiva@princeton.edu}{\texttt{varunsiva@princeton.edu}}.}
\date{}

\begin{document}
\begin{abstract}
We prove a general upper bound on the $k$-th adjacency eigenvalue of a graph. For $k\ge 2$, we show that
\[
\lambda_k(G)\le \frac{(k-2)\sqrt{k+1}+2}{2k(k-1)}\,n-1
\]
for every graph $G$ on $n$ vertices. We build on a recent approach that addresses the case $k=3$ and generalize the upper bound for all $k \geq 3$ by using the positivity of Gegenbauer polynomials. The upper bound is tight for $k \in \{2,3,4,8,24\}$. We also highlight the close relation of $\lambda_k(G)$ to questions about equiangular lines.
\end{abstract}

\maketitle

\section{Introduction}

For a simple undirected graph $G$ on $n$ vertices, let
\[
\lambda_1(G)\ge \lambda_2(G)\ge \cdots\ge \lambda_n(G)
\]
denote the eigenvalues of its adjacency matrix. Following Nikiforov~\cite{Nikiforov15}, define
\[
 c_k:=\sup\left\{\frac{\lambda_k(G)}{|V(G)|}: |V(G)|\ge k\right\}.
\]
The problem of bounding $\lambda_k(G)$ in terms of the order of the graph goes back at least to work of Hong and Powers~\cite{Hong88,Hong93,Powers89}. Nikiforov~\cite{Nikiforov15} proved the universal estimate
\[
 c_k\le \frac{1}{2\sqrt{k-1}}\qquad (k\ge 2),
\]
On the lower-bound side, the balanced complete $k$-partite graph gives $c_k\ge 1/k$, and Nikiforov showed that in fact $c_k>1/k$ for every $k\ge 5$~\cite{Nikiforov15}. He therefore asked whether $c_3=1/3$ and $c_4=1/4$~\cite[Question~2.11]{Nikiforov15}. Powers had earlier proposed a conjecture that $\lambda_k(G) \leq \lfloor \frac{n}{k} \rfloor$ for connected graphs, but Nikiforov later observed that the original proof is flawed; moreover, Nikiforov's constructions show that this bound fails for every $k\ge 5$, while Linz later ruled out the case $k=4$~\cite{Powers89,Nikiforov15,Linz23}. The remaining case $k=3$ is also recorded in the survey of Liu and Ning~\cite[Problem~10]{LiuNing23}.

Recent work has clarified the first unresolved cases. Leonida and Li verified the sharp third-eigenvalue bound for several important graph classes and formulated a weighted matrix conjecture that would imply it in full generality~\cite{LeonidaLi25}. See also Li~\cite{Li25} for a related asymptotic improvement on Nikiforov's bound. Very recently, Tang proved the sharp inequality
$\lambda_3(G)\le \frac n3-1$ for every graph $G$ by confirming the conjecture of Leonida and Li~\cite{Tang26}.

On the lower-bound side, it is easy to see that $c_3 \geq 1/3$ by taking the disjoint union of three equal sized cliques. \cite{Nikiforov15} proved that $c_k > \frac{1}{2\sqrt{k-1}+1}$ when $q$ is an odd prime power and $k = q^2 - q + 1$. \cite{Nikiforov15} also showed that $c_k > \frac{1}{2\sqrt{k-1}+k^{1/3}}$ for sufficiently large $k$. Further, Linz provided several improved constructions for $4 \leq k \leq 24$~\cite{Linz23}.

Our main result is the following.

\begin{theorem}\label{thm:main-general}
Let $k\ge 2$ be an integer. For every graph $G$ on $n$ vertices,
\[
\lambda_k(G)\le \alpha_k n-1,
\qquad
\alpha_k:= \frac{(k-2)\sqrt{k+1}+2}{2k(k-1)}
\]
Consequently, $c_k \leq \alpha_k$
\end{theorem}

For $k=2$, this simply follows from the classical sharp bound $\lambda_2(G)\le n/2-1$~\cite{Hong88,Nikiforov15}. Our result holds for $k \geq 3$. For $k=3$, it recovers Tang's sharp bound $\lambda_3(G)\le n/3-1$~\cite{Tang26}. For $k \in \{4,8,24\}$, $\alpha_k$ matches Linz's lower bound~\cite{Linz23}. Therefore, our upper bound is tight for $k \in \{2,3,4,8,24\}$. It is very interesting to note that our upper bound for $\lambda_k(G)$ is tight whenever an extremal equiangular lines configuration exists in dimension $r=k-1$. We discuss this in further detail in \Cref{disc:lowerbounds}.

For every $k\ge 2$ one has
$\alpha_k<\frac{1}{2\sqrt{k-1}}$, so \Cref{thm:main-general} improves Nikiforov's universal bound for $\lambda_k(G)$. We remark that Nikiforov claimed that there exists a positive $\epsilon_k$ such that $c_k \leq \frac{1}{2\sqrt{k-1}}-\epsilon_k$. However, the proof was omitted because it relied on the Removal Lemma of Alon,
Fischer, Krivelevich, and Szegedy \cite{alon2000efficient} with other tools of analytic graph theory. This would make the proof complicated and give a small $\epsilon_k$.

Our proof follows the overall strategy of Tang~\cite{Tang26}. As in their work, the graph-theoretic theorem is deduced from a weighted matrix inequality using the identity
\[
A(G)+A(\overline G)=J-I
\]
and Weyl's inequality. The key new modification is the use of Gegenbauer polynomials to generalize the trigonometric upper bound from the rank-$2$ case. The required positive-kernel input goes back to Schoenberg~\cite{Schoenberg42}; for related perspectives in spherical coding and semidefinite methods, see also Delsarte--Goethals--Seidel~\cite{DGS77} and Bachoc--Vallentin~\cite{BachocVallentin08}.

\section{Positive Kernels}

Fix an integer $r\ge 2$. Let $S^{r-1} = \{v \in \R^r: \|v\|_2 = 1\}$ be the unit sphere. For real vectors $u,v$, let $\inner{u}{v}$ denote the standard inner product. For real valued square matrices $A,B \in \R^{n\times n}$, write $\inner{A}{B}_F = \tr(A^T B)$ for the Frobenius inner product. This is equivalent to just taking taking the inner product by viewing $A$ and $B$ as $n^2$-dimensional vectors. For a matrix, $\|A\|_F = \sqrt{\froinner{A}{A}}$ and $\|A\|_1 = \sum_{i,j} |A_{ij}|$. Let $I_r \in \R^{r\times r}$ denote the identity matrix. A matrix $M \in \R^{n\times n}$ is positive semidefinite if $x^\top M x \geq 0$ for all $x \in \R^n$.

\begin{definition}
A function $f: \R \rightarrow \R$ is positive semidefinite on $S^{r-1}$ if for all unit vectors $u_1,\ldots,u_n \in S^{r-1}$, the matrix $M \in \R^{n\times n}$ given by $M_{ij} = f(\inner{u_i}{u_j})$ is positive semidefinite.
\end{definition}

Define 
\[f^r_2(t) = t^2 - \frac{1}{r}\]
\[f^r_4(t) = t^4 - \frac{3}{r(r+2)}\]

We will prove that $f_2^r$ and $f_4^r$ as positive semidefinite on $S^{r-1}$. In \cref{sec:l1}, this will play a crucial role in bounding $\|Q\|_1$ for a rank $r$ orthogonal projection matrix $Q \in \R^{n\times n}$.

\begin{lemma}\label{lem:psd2}
For $r \geq 2$, define
\[
\phi(u):=uu^\top-\frac1r I_r,
\qquad u\in S^{r-1}.
\]
Then for all $u,v\in S^{r-1}$,
\[
\froinner{\phi(u)}{\phi(v)}=f_2^r(\inner{u}{v}).
\]
Further, $f_2^r$ is positive semidefinite on $S^{r-1}$.
\end{lemma}

\begin{proof}
A direct computation gives
\begin{align*}
\froinner{\phi(u)}{\phi(v)}
&=\tr\!\left(\left(uu^\top-\frac1r I_r\right)\left(vv^\top-\frac1r I_r\right)\right) \\
&=(u^\top v)^2-\frac1r
 =\inner{u}{v}^2-\frac1r
 =f_2^r(\inner{u}{v}).
\end{align*}

Now observe that for any $u_1,\ldots,u_n \in S^{r-1}$, the $n \times n$ matrix $(f_2^r(\inner{u_i}{u_j}))_{i,j=1}^n$ is just the Gram matrix of the matrices $\phi(u_1),\ldots,\phi(u_n)$ with the standard Frobenius inner product. Therefore, it must be positive semidefinite. 
\end{proof}

To prove that $(f_2^r(\inner{u_i}{u_j}))_{i,j=1}^n$ is PSD, we expressed it as a Gram matrix of certain matrices with the Frobenius inner product. We will do the same thing for $f_4^r$, but since $f_4^r$ is a degree $4$ polynomial, we will need to work with $4$-tensors. The proof of \Cref{lem:psd4} is elementary but requires some computation. We defer this proof to the appendix.

\begin{lemma}\label{lem:psd4}
For $r \geq 2$,
\[
f_4^r(t)=t^4-\frac{3}{r(r+2)}
\]
is positive semidefinite on \(S^{r-1}\).
\end{lemma}

\begin{remark}
\Cref{lem:psd2} and \Cref{lem:psd4} are just special cases of a theorem by Schoenberg\cite{Schoenberg42}. Let $G_\ell^{\lambda}$ denote the Gegenbauer polynomial of degree $\ell$. These polynomials can be defined in terms of their generating function\cite{stein1971introduction}: 
\[\frac{1}{(1 - 2tx + x^2)^\lambda} = \sum_{\ell=0}^\infty G_{\ell}^\lambda (t) x^{\ell}\]
Schoenberg proved the following strong characterization: $f(t)$ is a real continuous function such that the matrix $f(\inner{u_i}{u_j})_{i,j=1}^n$ is positive semidefinite for all subsets $\{u_1,\ldots,u_n\}\subseteq S^{r-1}$ if and only if $f$ is of the form $\sum_{\ell=0}^\infty a_{\ell} G_{\ell}^{r/2-1}$ with $a_{\ell} \geq 0$.

In particular, $G_2^{r/2-1}$ and $G_4^{r/2-1}$ are positive semidefinite on $S^{r-1}$, and we are simply using $f_2^r = \frac{r-1}{r} G_2^{r/2-1}$ and $f_4^r = \frac{(r-1)(r+1)}{(r+2)(r+4)} G_4^{r/2-1} + \frac{6}{r+4} f_2^r$. Since $f_r^4$ and $f_r^4$ are just positive linear combinations of Gegenbauer polynomials with $\lambda = r/2-1$, they must be positive semidefinite on $S^{r-1}$ by \cite{Schoenberg42}. However, we do not require this strong characterization.  
\end{remark}

\section{An $\ell_1$ bound for rank-$r$ orthogonal projections}\label{sec:l1}

Let $Q=(q_{ij})\in \R^{n\times n}$ be a rank-$r$ orthogonal projection, where $r\ge 2$. So $Q$ satisfies $Q = Q^2 = Q^T$. The goal of this section is to determine an upper bound on $\|Q\|_1 = \sum_{i,j} |q_{ij}|$. It is easy to see that for a rank $r$ orthogonal projection $Q$, $\|Q\|_1 \leq \sqrt{n^2}\|Q\|_F = n \sqrt{r}$. In \cite{MagsinoMixonParshall2019}, they prove that $\|Q\|_1 \leq r + \sqrt{(n-1)r(n-r)}$. Our bound in \Cref{thm:l1-proj} is better precisely when $n \geq \binom{r+1}{2}$ with equality at $n = \binom{r+1}{2}$. We require this improved bound because $r$ will be small relative to $n$ in our applications. We prove the following:

\begin{theorem}\label{thm:l1-proj}
Let \(Q\in \R^{n\times n}\) be a rank-\(r\) orthogonal projection, where \(r\ge 2\). 
\[
\|Q\|_1=\sum_{i,j=1}^n \abs{q_{ij}}
\le \beta_r n \qquad \text{ where } \quad \beta_r:=\frac{r+\sqrt{r+2}}{1+\sqrt{r+2}}
\]
\end{theorem}
\begin{proof}
Since $Q$ is a rank $r$ orthogonal projection, we may write $Q = BB^\top$ such that $B \in \R^{n\times r}$ where the columns of $B$ form an orthogonal set in $\R^n$. Further, let $x_1,\ldots,x_n \in \R^r$ be the rows of $B$. Write $x_i=c_i u_i$, where $c_i=\norm{x_i}_2\ge 0$ and $u_i\in S^{r-1}$ whenever $c_i\ne 0$; if $c_i=0$, choose $u_i$ arbitrarily in $S^{r-1}$. Since the columns of $B$ are orthogonal, $B^\top B = I_r$. Writing out $B^\top B$ in terms of the rows of $B$, we obtain
\begin{equation}\label{eq:projection-model}
\sum_{i=1}^n c_i^2 u_i u_i^\top=I_r 
\qquad\text{and}\qquad
q_{ij}=c_i c_j\inner{u_i}{u_j}
\end{equation}


Set
\[
C:=\sum_{i=1}^n c_i,
\qquad
x:=\sum_{i=1}^n c_i\phi(u_i),
\qquad
X:=\norm{x}_F.
\]

To prove \Cref{thm:l1-proj}, we establish \Cref{lem:CS-ineq}, \Cref{lem:abstract-majorant} and \Cref{lem:scalar-majorant}.

\begin{lemma}\label{lem:CS-ineq}
With the notation above,
\[
C^2+rX^2\le rn.
\]
\end{lemma}

\begin{proof}
\begin{align*}
C^2 + rX^2
&= \sum_{i=1}^n \sum_{j=1}^n c_i c_j + r \sum_{i=1}^n \sum_{j=1}^n c_i c_j \froinner{\phi(u_i)}{\phi(u_j)}\\
&= \sum_{i=1}^n \sum_{j=1}^n c_i c_j + r \sum_{i=1}^n \sum_{j=1}^n c_i c_j \left(\inner{u_i}{u_j}^2 - \frac{1}{r}\right) \text{\qquad by \Cref{lem:psd2}}\\
&= r \sum_{i=1}^n \sum_{j=1}^n c_i c_j \inner{u_i}{u_j}^2\\
&\leq \frac{r}{2}\sum_{i=1}^n \sum_{j=1}^n (c_i^2 + c_j^2) \inner{u_i}{u_j}^2 \text{\qquad by Cauchy-Schwarz}\\
&= r \sum_{i=1}^n \sum_{j=1}^n c_j^2 \inner{u_i}{u_j}^2 \text{\qquad by symmetry}
\end{align*}

\begin{align*}
r \sum_{i=1}^n \sum_{j=1}^n c_j^2 \inner{u_i}{u_j}^2
&= r \sum_{i=1}^n \sum_{j=1}^n \tr(u_i u_i^\top c_j^2 u_j u_j^\top)\\
&= r \sum_{i=1}^n  \tr\left(u_i u_i^\top \sum_{j=1}^n c_j^2 u_j u_j^\top\right)\\
&= r \sum_{i=1}^n  \tr\left(u_i u_i^\top\right) \text{\qquad by \Cref{eq:projection-model}}\\
&= r n
\end{align*}
\end{proof}

We will now bound \(\|Q\|_1\) by finding a polynomial $h(t)$ that satisfies $|t| \leq h(t)$ for all $t \in [-1,1]$ and applying it entry wise to $|q_{ij}| = c_i c_j |\inner{u_i}{u_j}|$ with $t = |\inner{u_i}{u_j}|$. We will consider an expression of the form
\[
a+b\,f_2^r(t)-\gamma\,f_4^r(t).
\]
This is because after substituting \(t=\inner{u_i}{u_j}\) and summing over \(i,j\), the constant term produces $C^2$, the \(f_2^r\)-term produces \(X^2\), while the \(f_4^r\)-term is nonnegative since $f_4^r$ is positive semidefinite by \Cref{lem:psd4}. The following lemma formalizes this reduction.

\begin{lemma}\label{lem:abstract-majorant}
Let \(a,b,\gamma\in\R\) satisfy
\[
\abs{t}\le a+b\,f_2^r(t)-\gamma\,f_4^r(t)
\qquad\text{for all } t\in[-1,1],
\]
with \(\gamma\ge 0\) and \(b\le ra\). Then every rank-\(r\) orthogonal projection \(Q\in\R^{n\times n}\) satisfies
\[
\|Q\|_1\le arn.
\]
\end{lemma}

\begin{proof}
Recall by \eqref{eq:projection-model} that we may write $Q = BB^T$ where the rows of $B$ are given by $c_i u_i$ for $i \in [n]$ and $\|Q\|_1 = \sum_{i,j} |q_{ij}| = \sum_{i,j} c_i c_j |\inner{u_i}{u_j}|$. Applying the assumed scalar inequality with \(t=\inner{u_i}{u_j}\), using the fact that $c_i \geq 0$ and summing over all \(i,j\), we obtain
\[
\|Q\|_1 = \sum_{i,j} |q_j|
\le a\sum_{i,j} c_i c_j
+b\sum_{i,j} c_i c_j f_2^r(\inner{u_i}{u_j})
-\gamma\sum_{i,j} c_i c_j f_4^r(\inner{u_i}{u_j}).
\]
By definition,
\[
\sum_{i,j} c_i c_j=C^2,
\qquad
\sum_{i,j} c_i c_j f_2^r(\inner{u_i}{u_j}) = \sum_{i,j} c_i c_j \froinner{\phi(u_i)}{\phi(u_j)} = X^2
\]
Also, the positive semidefiniteness of \(\bigl(f_4^r(\inner{u_i}{u_j})\bigr)_{i,j=1}^n\) implies that
\[
\sum_{i,j} c_i c_j f_4^r(\inner{u_i}{u_j})\ge 0.
\]
Therefore
\[
\|Q\|_1\le aC^2+bX^2.
\]
Since \(b\le ra\), we have
\[
aC^2+bX^2\le a\bigl(C^2+rX^2\bigr).
\]
Now apply \Cref{lem:CS-ineq} to get
\[
\|Q\|_1\le arn.
\]
\end{proof}

We now exhibit a concrete choice of coefficients for which the hypotheses of \Cref{lem:abstract-majorant} hold. Clearly, we would like $ar$ to be as small as possible. The choice of coefficients in \Cref{lem:scalar-majorant} below may seem mysterious but it is motivated by a connection to equiangular lines. A more detailed discussion is included in \Cref{sec5}. 


\begin{lemma}\label{lem:scalar-majorant}
Let \(r\ge 2\), and set
\[
s:=\sqrt{r+2}.
\]
Define
\[
a_r:=\frac{s^2+s-2}{r(s+1)},\qquad
b_r:=\frac{s(s^2+2s+3)}{2(s+1)^2},\qquad
\gamma_r:=\frac{s^3}{2(s+1)^2}.
\]
Then for every \(t\in[-1,1]\),
\[
|t|\le a_r+b_r f_2^r(t)-\gamma_r f_4^r(t).
\]
Moreover,
\[
b_r\le ra_r.
\]
\end{lemma}

\begin{proof}
Recall that
\[
f_2^r(t)=t^2-\frac1r,\qquad
f_4^r(t)=t^4-\frac{3}{r(r+2)}.
\]
Since the left-hand side is even in \(t\), it suffices to consider \(t\in[0,1]\).
A direct expansion gives
\[
a_r+b_r f_2^r(t)-\gamma_r f_4^r(t)-t
=
\frac{(1-t)(st-1)^2(st+s+2)}{2(s+1)^2}.
\]
The right-hand side is nonnegative for \(t\in[0,1]\), since
\[
1-t\ge 0,\qquad (st-1)^2\ge 0,\qquad st+s+2>0.
\]
Hence
\[
|t|\le a_r+b_r f_2^r(t)-\gamma_r f_4^r(t)
\qquad\text{for all } t\in[-1,1].
\]

Finally,
\[
ra_r-b_r
=
\frac{s^3+2s^2-5s-4}{2(s+1)^2}
=
\frac{(s-2)(s+1)(s+3)+2}{2(s+1)^2}\ge 0,
\]
and therefore \(b_r\le ra_r\).
\end{proof}

\noindent\textit{Completing the proof of \cref{thm:l1-proj}}.

Apply \Cref{lem:abstract-majorant} with
\[
a=a_r,\qquad b=b_r,\qquad \gamma=\gamma_r.
\]
By \Cref{lem:scalar-majorant}, the hypotheses of \Cref{lem:abstract-majorant} are satisfied. Therefore
\[
\|Q\|_1\le a_r r n = \frac{s^2 + s -2}{s+1} n = \frac{r+\sqrt{r+2}}{1+\sqrt{r+2}} n = \beta_r n.
\]

\end{proof}

\section{From projections to eigenvalues}

\begin{theorem}\label{thm:bottom-r}
Let $A=(a_{ij})\in \R^{n\times n}$ be symmetric, with eigenvalues
\[
\mu_1\ge \mu_2\ge \cdots\ge \mu_n.
\]
Assume that
\[
0\le a_{ij}\le 1 \quad (i\ne j),
\qquad
 a_{ii}\ge 0 \quad (1\le i\le n).
\]
Let $r\ge 2$ be an integer. Then
\[
\mu_{n-r+1}+\cdots+\mu_n\ge -\frac{\beta_r}{2}\,n.
\]
In particular,
\[
\mu_{n-r+1}\ge -\frac{\beta_r}{2r}\,n.
\]
\end{theorem}

\begin{proof}
Let $\cP_r$ denote the set of rank-$r$ orthogonal projections in $\R^{n\times n}$. By the variation characterization of eigenvalues (Ky Fan's principle), it is easy to see that \[\lambda_1(M) + \cdots + \lambda_r(M) = \max_{B \in \R^{n\times r}: B^T B =  I_r} \tr(B^T M B) = \max_{B \in \R^{n\times r}: B^T B =  I_r} \tr(M BB^T) = \max_{Q \in \mathcal{P}_r} \tr(M Q)\] We simply apply this to $M = -A$ to obtain: 
\[
\mu_{n-r+1}+\cdots+\mu_n=
\min_{Q\in\cP_r} \tr(AQ).
\]

Fix $Q=(q_{ij})\in \cP_r$. Since $Q$ is positive semidefinite, we have $\1^\top Q\1\ge 0$. Further, $\tr(Q)= \tr(BB^T) = \tr(B^TB) = r$. This is also evident because $Q$ is a rank $r$ orthogonal projection and so has $r$ eigenvectors with eigenvalue $1$.
\[
\sum_{i<j} q_{ij}=\frac{\1^\top Q\1-r}{2}\ge -\frac r2.
\]
Also,
\[
\sum_{i<j}\abs{q_{ij}}=\frac{\|Q\|_1-r}{2}.
\]
Hence
\begin{align*}
\sum_{i<j}\min(q_{ij},0)
&=\frac12\left(\sum_{i<j}q_{ij}-\sum_{i<j}\abs{q_{ij}}\right) \\
&\ge \frac12\left(-\frac r2-\frac{\|Q\|_1-r}{2}\right)
 =-\frac{\|Q\|_1}{4}
 \ge -\frac{\beta_r}{4}n,
\end{align*}
where the last step uses \Cref{thm:l1-proj}.

Now $a_{ij}q_{ij}\ge \min(q_{ij},0)$ for every $i<j$, because $0\le a_{ij}\le 1$, and $a_{ii}q_{ii}\ge 0$ because both factors are nonnegative. Therefore
\begin{align*}
\tr(AQ)
&=\sum_i a_{ii}q_{ii}+2\sum_{i<j} a_{ij}q_{ij} \\
&\ge 2\sum_{i<j}\min(q_{ij},0)
 \ge -\frac{\beta_r}{2}n.
\end{align*}
Taking the minimum over $Q\in \cP_r$ proves the first claim.

Since $\mu_{n-r+1}\ge \mu_{n-r+2}\ge \cdots\ge \mu_n$, we have
\[
r\,\mu_{n-r+1}\ge \mu_{n-r+1}+\cdots+\mu_n,
\]
so
\[
\mu_{n-r+1}\ge -\frac{\beta_r}{2r}n.
\]
\end{proof}

\begin{proof}[Proof of \Cref{thm:main-general}]
The case $k=2$ is the classical sharp bound $\lambda_2(G)\le n/2-1$~\cite{Hong88,Nikiforov15}. So assume $k\ge 3$, and set
\[
r:=k-1.
\]
Let $\overline G$ be the complement of $G$. Then
\[
A(G)+A(\overline G)=J-I.
\]
The eigenvalues of $J-I$ are $n-1,-1,\dots,-1$, so in particular
\[
\lambda_2(J-I)=-1.
\]
We use Weyl's inequality in the form
\[
\lambda_i(X)+\lambda_j(Y)\le \lambda_{i+j-n}(X+Y)
\qquad\text{whenever } i+j\ge n+1.
\]
Apply this with
\[
X=A(G),
\qquad
Y=A(\overline G),
\qquad
i=k,
\qquad
j=n-k+2.
\]
Since $i+j=n+2$, we obtain
\[
\lambda_k(G)+\lambda_{n-k+2}(\overline G)\le \lambda_2(J-I)=-1.
\]
Now apply \Cref{thm:bottom-r} with $r=k-1$ to $A(\overline G)$ to get
\[
\lambda_{n-k+2}(\overline G)\ge -\frac{\beta_{k-1}}{2(k-1)}n.
\]
Since
\[
\frac{\beta_{k-1}}{2(k-1)}
=\frac{k-1+\sqrt{k+1}}{2(k-1)(\sqrt{k+1}+1)}
= \frac{(k-2)\sqrt{k+1}+2}{2k(k-1)}
=\alpha_k,
\]
it follows that
\[
\lambda_{n-k+2}(\overline G)\ge -\alpha_k n
\]
Therefore
\[
\lambda_k(G)\le \alpha_k n - 1
\]
which is the desired bound.
\end{proof}

\section{Concluding Remarks}\label{sec5}
\subsection{Tight Lower Bounds from Equiangular Lines}\label{disc:lowerbounds}
A set of lines in $\R^r$ passing through the origin is said to be equiangular if every pair of distinct lines determine the same angle. Let $N$ be the maximum number of equiangular lines in $\R^r$. Gerzon proved that $N \leq \binom{r+1}{2}$(see \cite{lemmens1973equiangular}). We call a set of equiangular lines extremal if $N = \binom{r+1}{2}$. We remark that the argument below is not new, we simply apply it to our problem. 

Fix $r \geq 2$. Suppose there exists a system of $N = \binom{r+1}{2}$ equiangular lines in $\R^r$. Let $k=r+1$. Then there exist a graph on $n = 2N$ vertices such that $\alpha_k(G) = \alpha_k n - 1$. 

Let $u_1,\ldots,u_N$ be the representative unit vectors of the equiangular lines. Then $\inner{u_i}{u_j} \in \{-\alpha,\alpha\}$ for some $\alpha \in (0,1)$. \cite{lemmens1973equiangular} proved that if $N = \binom{r+1}{2}$, then $\alpha = \frac{1}{\sqrt{r+2}}$. Using this, we can construct a graph $G$ that satisfies $\lambda_k(G) = \alpha_k n - 1$. This construction goes back to Seidel’s theory of regular two-graphs\cite{Seidel1976SurveyTwoGraphs}.

Let $V(G) = \{(i,+),(i,-): i \in [N]\}$. So $n = |V(G)| = 2n$. For $i \neq j$, if $\inner{x_i}{x_j} = \alpha$, put edges between $(i,+)$ and $(j,+)$ and between $(i,-)$ and $(j,-)$. For $i \neq j$, if $\inner{x_i}{x_j} = -\alpha$, then put edges between $(i,+)$ and $(j,-)$ and between $(i,-)$ and $(j,+)$. The spectrum of such a graph is well understood (see \cite{GodsilHensel1992}) and in particular, the $r+1 = k\text{th}$ eigenvalue is $\frac{N-r}{\alpha r}$. Substituting $n = 2N$, $\alpha = \frac{1}{\sqrt{r+2}}$ and $r = k-1$ yields $\alpha_k n - 1$. 

So our upper bound is tight for every $k$ when an extremal equiangular lines construction exists for $r = k-1$. However, such a construction is only known to exist for $r\in \{1,2,3,7,23\}$. Therefore, we only have tight results for $k \in \{1,2,3,4,8,24\}$. This is also reminiscent of the sphere packing problem where tight results are known for $d \in \{1,2,3,8,24\}$. This is not a coincidence: it is known that the $E_8$ lattice is optimal for $d = 8$\cite{viazovska2017sphere} and the Leech lattice is optimal for $d=24$\cite{cohn2017sphere}. The extremal equiangular lines constructions for $r=7$ and $23$ are also derived from these lattices\cite{gillespie2018equiangular}.

However, there may be graphs that match our upper bound for $\lambda_k(G)$ that do not arise from extremal equiangular constructions. It is also possible that our bound can be improved when an extremal equiangular construction does not exist. 

\subsection{Multiplicity of Second Eigenvalue}
For a graph $G$, let $m$ be the multiplicity of $\lambda_2(G)$. Therefore, $\lambda_2 = \lambda_{m+1} \leq \alpha_{m+1} n - 1$. 
Equivalently,
\[
m \le \max\bigl\{i\ge 1:\ \lambda_2(G)+1\le \alpha_{i+1}n\bigr\}.
\]
So our result gives an implicit upper bound on the second-eigenvalue multiplicity in terms of \(\lambda_2(G)\) and \(n\). By the same graph construction as in \Cref{disc:lowerbounds}, this is sharp whenever $m+1 \in \{2,3,4,8,24\}$. Since $\frac{1}{\alpha_{m+1}} = \Theta(\sqrt{m})$, we obtain $m \leq O((\frac{n}{\lambda_2 + 1})^2)$. The asymptotic result already follows from \cite{Nikiforov15} but we actually get sharper bounds. 

\cite{balla2024equiangular} show that $m$ is bounded by roughly $\frac{n}{\lambda_2 + 1}$ for graphs with maximum degree $\Delta$ where $\Delta^{O(1)} \leq \log n$ in order to obtain an exact bound on the number of equiangular lines in $\R^r$ with a fixed angle $\arccos{1/(2k-1)}$ for $r$ at least exponential in a polynomial of $k$, improving on the breakthrough result in \cite{jiang2021equiangular}. The bound in \cite{balla2024equiangular} focuses on the sparse graph regime and gets progressively weaker as $\Delta$ becomes polynomial while our bound remains meaningful even for much denser graphs. 

It would be interesting to obtain bounds on $\lambda_k(G)$ for graphs that are sparse or have bounded degree. Here is a simple bound. If $G$ has average degree $d$ and number of edges $e = nd/2$, then $\tr(A(G)^2) = 2e = \sum_{i=1}^n \lambda_i^2$. So $k \lambda_k^2 \leq 2e = nd$, so $\lambda_k \leq \sqrt{nd/k}$. Obtaining better bounds would be interesting for various regimes of $d$.

\subsection{Improving the Upper Bound}
Extremal constructions for the equiangular lines problem with $N = \binom{r+1}{2}$ can only exist when $r=2,3$ or $r+2$ is an odd perfect square\cite{lemmens1973equiangular}. Even when $r+2$ is an odd perfect square, the question of existence remains wide open. (We do know that such constructions do not exist for infinitely many $r$ with $r+2$ an odd perfect square \cite{NebeVenkov2013}\cite{BannaiMunemasaVenkov2005}). Therefore, obtaining a general improvement on our upper bound even by an additive constant for general $k = r+1$ is a difficult problem because it would immediately disprove the existence of extremal equiangular lines constructions in $\R^r$. Improvements for $r$ where $r+2$ is not an odd perfect square are likely more feasible. It would be interesting to study the very concrete problem of finding a tight upper bound for $\lambda_5(G)$ to understand the landscape better. 

\subsection{Motivation for Choice of Polynomial}\label{motivation}
In \Cref{lem:scalar-majorant}, we we find coefficients $a_r,b_r,\gamma_r$ such that $H(t) = a_r+b_r f_2^r(t)-\gamma_r f_4^r(t)-t \geq 0$ for $t \in [-1,1]$ with $r a_r$ as small as possible. It is reasonable to try and find $H$ such that $H(1) = 0$ because if our bound is not tight at the endpoint, we could potentially improve the bound further. Next, for the extremal equiangular constructions with $N = \binom{r+1}{2}$, the common angle is $\alpha = \frac{1}{\sqrt{r+2}} = \frac{1}{s}$. Since we apply the polynomial to bound terms of the form $|t| = |\inner{u_i}{u_j}|$ which is the common angle between two unit vectors,  a reasonable heuristic is to find $H$ such that $H(1/s) = 0$. Since $1/s$ is in the interior of $[0,1]$, we also set $H'(1/s) = 0$. This turns out to uniquely define $H$. \\

\textbf{Acknowledgments}:
We thank Noga Alon, Matija Bucić, Sergio Cristancho, Julien Codsi, Alex Divoux, Andrew Lin and Shouda Wang for helpful discussions. This paper was written with the help of ChatGPT 5.4.

\providecommand{\MR}[1]{}
\providecommand{\MRhref}[2]{%
  \href{http://www.ams.org/mathscinet-getitem?mr=#1}{#2}
}

   \bibliography{ref}

@article{BachocVallentin08,
  author  = {C. Bachoc and F. Vallentin},
  title   = {New upper bounds for kissing numbers from semidefinite programming},
  journal = {Journal of the American Mathematical Society},
  volume  = {21},
  number  = {3},
  pages   = {909--924},
  year    = {2008},
  doi     = {10.1090/S0894-0347-07-00589-9}
}

@article{DGS77,
  author  = {P. Delsarte and J. M. Goethals and J. J. Seidel},
  title   = {Spherical codes and designs},
  journal = {Geometriae Dedicata},
  volume  = {6},
  number  = {3},
  pages   = {363--388},
  year    = {1977},
  doi     = {10.1007/BF03187604}
}

@article{Hong88,
  author  = {Y. Hong},
  title   = {Bounds of eigenvalues of a graph},
  journal = {Acta Mathematicae Applicatae Sinica},
  volume  = {4},
  number  = {2},
  pages   = {165--168},
  year    = {1988},
  doi     = {10.1007/BF02006065}
}

@article{Hong93,
  author  = {Y. Hong},
  title   = {Bounds of eigenvalues of graphs},
  journal = {Discrete Mathematics},
  volume  = {123},
  number  = {1--3},
  pages   = {65--74},
  year    = {1993},
  doi     = {10.1016/0012-365X(93)90007-G}
}

@article{LeonidaLi25,
  author  = {G. Leonida and S. Li},
  title   = {On graphs with large third eigenvalue},
  journal = {arXiv preprint arXiv:2501.02563},
  year    = {2025},
  doi     = {10.48550/arXiv.2501.02563},
  eprint  = {2501.02563},
  archivePrefix = {arXiv},
  primaryClass  = {math.CO},
  url     = {https://arxiv.org/abs/2501.02563}
}

@article{Li25,
  author  = {S. Li},
  title   = {Strengthened upper bound on the third eigenvalue of graphs},
  journal = {arXiv preprint arXiv:2501.07494},
  year    = {2025},
  doi     = {10.48550/arXiv.2501.07494},
  eprint  = {2501.07494},
  archivePrefix = {arXiv},
  primaryClass  = {math.CO},
  url     = {https://arxiv.org/abs/2501.07494}
}

@article{Linz23,
  author  = {W. Linz},
  title   = {Improved Lower Bounds on the Extrema of Eigenvalues of Graphs},
  journal = {Graphs and Combinatorics},
  volume  = {39},
  number  = {4},
  pages   = {82},
  year    = {2023},
  doi     = {10.1007/s00373-023-02678-0}
}

@article{LiuNing23,
  author  = {L. Liu and B. Ning},
  title   = {Unsolved problems in spectral graph theory},
  journal = {Operations Research Transactions},
  volume  = {27},
  number  = {4},
  pages   = {33--60},
  year    = {2023},
  doi     = {10.15960/j.cnki.issn.1007-6093.2023.04.003}
}

@article{Nikiforov15,
  author  = {V. Nikiforov},
  title   = {Extrema of graph eigenvalues},
  journal = {Linear Algebra and its Applications},
  volume  = {482},
  pages   = {158--190},
  year    = {2015},
  doi     = {10.1016/j.laa.2015.05.016}
}

@article{Powers89,
  author  = {D. L. Powers},
  title   = {Bounds on graph eigenvalues},
  journal = {Linear Algebra and its Applications},
  volume  = {117},
  pages   = {1--6},
  year    = {1989},
  doi     = {10.1016/0024-3795(89)90541-7}
}

@article{Schoenberg42,
  author  = {I. J. Schoenberg},
  title   = {Positive definite functions on spheres},
  journal = {Duke Mathematical Journal},
  volume  = {9},
  number  = {1},
  pages   = {96--108},
  year    = {1942},
  doi     = {10.1215/S0012-7094-42-00908-6}
}

@article{Tang26,
  author  = {Q. Tang},
  title   = {A sharp upper bound on the third adjacency eigenvalue of a graph},
  journal = {arXiv preprint arXiv:2603.21181},
  year    = {2026},
  doi     = {10.48550/arXiv.2603.21181},
  eprint  = {2603.21181},
  archivePrefix = {arXiv},
  primaryClass  = {math.CO},
  url     = {https://arxiv.org/abs/2603.21181}
}

@article{alon2000efficient,
  title={Efficient testing of large graphs},
  author={N. Alon and E. Fischer and M. Krivelevich and M. Szegedy},
  journal={Combinatorica},
  volume={20},
  number={4},
  pages={451--476},
  year={2000},
  publisher={Springer}
}

@article{lemmens1973equiangular,
  title={Equiangular lines},
  author={P. W. H. Lemmens and J. J. Seidel},
  journal={Journal of Algebra},
  volume={24},
  number={3},
  pages={494--512},
  year={1973},
  publisher={Elsevier}
}

@inproceedings{Seidel1976SurveyTwoGraphs,
  author    = {J. J. Seidel},
  title     = {A Survey of Two-Graphs},
  booktitle = {Colloquio Internazionale sulle Teorie Combinatorie (Rome, 1973), Tomo I},
  series    = {Atti dei Convegni Lincei},
  volume    = {17},
  pages     = {481--511},
  year      = {1976},
  publisher = {Accademia Nazionale dei Lincei},
  address   = {Rome}
}

@article{GodsilHensel1992,
  author  = {C. D. Godsil and A. D. Hensel},
  title   = {Distance regular covers of the complete graph},
  journal = {Journal of Combinatorial Theory, Series B},
  volume  = {56},
  number  = {2},
  pages   = {205--238},
  year    = {1992},
  doi     = {10.1016/0095-8956(92)90019-T}
}

@article{balla2024equiangular,
  title={Equiangular lines via improved eigenvalue multiplicity},
  author={I. Balla and M. Buci{\'c}},
  journal={arXiv preprint arXiv:2409.16219},
  year={2024}
}

@article{jiang2021equiangular,
  title={Equiangular lines with a fixed angle},
  author={Z. Jiang and J. Tidor and Y. Yao and S. Zhang and Y. Zhao},
  journal={Annals of Mathematics},
  volume={194},
  number={3},
  pages={729--743},
  year={2021},
  publisher={Department of Mathematics, Princeton University Princeton, New Jersey, USA}
}

@article{viazovska2017sphere,
  title={The sphere packing problem in dimension 8},
  author={M. Viazovska},
  journal={Annals of mathematics},
  pages={991--1015},
  year={2017},
  publisher={JSTOR}
}

@article{cohn2017sphere,
  title={The sphere packing problem in dimension 24},
  author={H. Cohn and A. Kumar and S. Miller and D. Radchenko and M. Viazovska},
  journal={Annals of mathematics},
  volume={185},
  number={3},
  pages={1017--1033},
  year={2017},
  publisher={Department of Mathematics, Princeton University Princeton, New Jersey, USA}
}

@article{gillespie2018equiangular,
  title={Equiangular lines, incoherent sets and quasi-symmetric designs},
  author={N. Gillespie},
  journal={arXiv preprint arXiv:1809.05739},
  year={2018}
}

@article{MagsinoMixonParshall2019,
  author     = {M. Magsino and D. Mixon and H. Parshall},
  title      = {A Delsarte-Style Proof of the Bukh-Cox Bound},
  journal    = {CoRR},
  volume     = {abs/1902.00926},
  year       = {2019},
  eprinttype = {arXiv},
  eprint     = {1902.00926},
  url        = {http://arxiv.org/abs/1902.00926}
}

@book{stein1971introduction,
  title={Introduction to Fourier analysis on Euclidean spaces},
  author={E. Stein and G. Weiss},
  volume={1},
  year={1971},
  publisher={Princeton university press}
}

@article{BannaiMunemasaVenkov2005,
  author  = {E. Bannai and A. Munemasa and B. Venkov},
  title   = {The nonexistence of certain tight spherical designs},
  journal = {St. Petersburg Mathematical Journal},
  volume  = {16},
  number  = {4},
  pages   = {609--625},
  year    = {2005},
  doi     = {10.1090/S1061-0022-05-00868-X}
}

@article{NebeVenkov2013,
  author  = {G. Nebe and B. Venkov},
  title   = {On tight spherical designs},
  journal = {St. Petersburg Mathematical Journal},
  volume  = {24},
  number  = {3},
  pages   = {485--491},
  year    = {2013},
  doi     = {10.1090/S1061-0022-2013-01249-0}
}
   \bibliographystyle{amsalpha}

\pagebreak

\section*{Appendix}

We prove \Cref{lem:psd4} here. We will make use of $4$-tensors. A $4$-tensor $T=(T_{abcd})_{a,b,c,d=1}^r$ is simply an array of real numbers with four indices.  
If $T=(T_{abcd})$ and $S=(S_{abcd})$ are two $4$-tensors of the same size, their Frobenius inner product is defined by
\[
\froinner{T}{S}:=\sum_{a,b,c,d=1}^r T_{abcd}S_{abcd}.
\]
This is the natural analogue of the usual Frobenius inner product for matrices.  
For a vector $u=(u_1,\dots,u_r)\in\mathbb R^r$, we write $u^{\otimes 4}$ for the $4$-tensor with entries
\[
(u^{\otimes 4})_{ijkl}:=u_i u_j u_k u_l.
\]
A $4$-tensor is called symmetric if its entries do not change when the indices are permuted.

\subsection*{\Cref{lem:psd4}} For $r \geq 2$,
\[
f_4^r(t)=t^4-\frac{3}{r(r+2)}
\]
is positive semidefinite on \(S^{r-1}\).

\begin{proof}
Let \(\Omega=(\Omega_{ijkl})_{1\le i,j,k,l\le r}\) be the symmetric \(4\)-tensor
\[
\Omega_{ijkl}:=\frac{1}{r(r+2)}
\bigl(\delta_{ij}\delta_{kl}+\delta_{ik}\delta_{jl}+\delta_{il}\delta_{jk}\bigr),
\]
where \(\delta_{ab}\) is the Kronecker delta, that is,
\[
\delta_{ab}=
\begin{cases}
1,& a=b,\\
0,& a\neq b.
\end{cases}
\]
For \(u\in S^{r-1}\), define
\[
\psi(u):=u^{\otimes 4}-\Omega.
\]
We claim that for all \(u,v\in S^{r-1}\),
\[
\froinner{\psi(u)}{\psi(v)}=f_4^r(\inner{u}{v}).
\]
Write \(t=\inner{u}{v}\). We compute each term separately.

First,
\[
\froinner{u^{\otimes 4}}{v^{\otimes 4}}
=\sum_{i,j,k,l} u_i u_j u_k u_l\, v_i v_j v_k v_l
=(u^\top v)^4=t^4.
\]

Next,
\begin{align*}
\froinner{u^{\otimes 4}}{\Omega}
&=\frac{1}{r(r+2)}
\sum_{i,j,k,l}u_i u_j u_k u_l
\bigl(\delta_{ij}\delta_{kl}+\delta_{ik}\delta_{jl}+\delta_{il}\delta_{jk}\bigr).
\end{align*}
Each of the three terms contributes
\[
\sum_{i,j,k,l}u_i u_j u_k u_l\,\delta_{ij}\delta_{kl}
=\sum_{i,k}u_i^2u_k^2
=\Bigl(\sum_i u_i^2\Bigr)^2
=1,
\]
so
\[
\froinner{u^{\otimes 4}}{\Omega}=\frac{3}{r(r+2)}.
\]
Similarly,
\[
\froinner{v^{\otimes 4}}{\Omega}=\frac{3}{r(r+2)}.
\]

It remains to compute \(\froinner{\Omega}{\Omega}\). Let
\[
C_{ijkl}:=\delta_{ij}\delta_{kl}+\delta_{ik}\delta_{jl}+\delta_{il}\delta_{jk},
\]
so that \(\Omega=\frac{1}{r(r+2)}C\). Then
\[
\froinner{C}{C}
=\sum_{i,j,k,l} C_{ijkl}^2.
\]
Expanding \(C_{ijkl}^2\), the three square terms contribute
\[
\sum_{i,j,k,l}(\delta_{ij}\delta_{kl})^2
=\sum_{i,j,k,l}\delta_{ij}\delta_{kl}
=r^2,
\]
and likewise for the other two square terms, giving \(3r^2\) in total.

For the mixed terms,
\[
\sum_{i,j,k,l}\delta_{ij}\delta_{kl}\delta_{ik}\delta_{jl}=r,
\]
since all four indices must be equal. The same holds for each of the other five mixed terms. Hence
\[
\froinner{C}{C}=3r^2+6r=3r(r+2),
\]
and therefore
\[
\froinner{\Omega}{\Omega}
=\frac{1}{r^2(r+2)^2}\,\froinner{C}{C}
=\frac{3}{r(r+2)}.
\]

Putting everything together,
\begin{align*}
\froinner{\psi(u)}{\psi(v)}
&=\froinner{u^{\otimes 4}-\Omega}{v^{\otimes 4}-\Omega} \\
&=\froinner{u^{\otimes 4}}{v^{\otimes 4}}
-\froinner{u^{\otimes 4}}{\Omega}
-\froinner{v^{\otimes 4}}{\Omega}
+\froinner{\Omega}{\Omega} \\
&=t^4-\frac{3}{r(r+2)}-\frac{3}{r(r+2)}+\frac{3}{r(r+2)} \\
&=t^4-\frac{3}{r(r+2)} \\
&=f_4^r(t).
\end{align*}

Thus, for any \(u_1,\dots,u_n\in S^{r-1}\), the matrix
\[
\bigl(f_4^r(\inner{u_i}{u_j})\bigr)_{i,j=1}^n
=
\bigl(\froinner{\psi(u_i)}{\psi(u_j)}\bigr)_{i,j=1}^n
\]
is a Gram matrix, hence is positive semidefinite.
\end{proof}


\end{document}